\documentclass[12pt]{amsart}
\usepackage{amssymb, colordvi}
\usepackage{multirow}
\usepackage{graphicx}

\newtheorem{theorem}{Theorem}
\newtheorem{proposition}[theorem]{Proposition}

\newtheorem{ex}[theorem]{Example}

\newenvironment{example}{\begin{ex}\rm}{\end{ex}}

\newcounter{FNC}[page]
\def\fauxfootnote#1{{\addtocounter{FNC}{2}$^\fnsymbol{FNC}$%
     \let\thefootnote\relax\footnotetext{$^\fnsymbol{FNC}$#1}}}

\newcommand{\calV}{\mathcal{V}}

\newcommand{\R}{\mathbb{R}}

\newcommand{\C}{\mathbb{C}}
\newcommand{\Log}{\mbox{Log}}

\title[Betti numbers of fewnomial hypersurfaces]{Betti number bounds for
  fewnomial hypersurfaces via stratified Morse Theory} 

\author{Fr\'ed\'eric Bihan}
\address{Laboratoire de Math\'ematiques\\
         Universit\'e de Savoie\\
         73376 Le Bourget-du-Lac Cedex\\
         France}

\email{Frederic.Bihan@univ-savoie.fr}
\urladdr{http://www.lama.univ-savoie.fr/\~{}bihan/}

\author{Frank Sottile}
\address{Department of Mathematics\\
         Texas A\&M University\\
         College Station\\
         Texas \ 77843\\
         USA}
\email{sottile@math.tamu.edu}
\urladdr{http://www.math.tamu.edu/\~{}sottile/}
\thanks{Sottile supported by NSF CAREER grant DMS-0538734 and NSF
  grant DMS-0701050}  


\begin{document}

\begin{abstract}
  We use stratified Morse theory for a manifold with
  corners to give a new bound for the sum of the Betti numbers of a fewnomial hypersurface
  in $\R^N_{>}$.
\end{abstract}
\maketitle

In the book ``Fewnomials''~\cite{Kh91}, A. Khovanskii gives 
bounds on the Betti numbers of varieties $X$ in $\R^n$ or in the
positive orthant $\R^n_{>}$. 
These varieties include algebraic varieties, 
varieties defined by polynomial functions in the variables and
exponentials of the variables, and also by even more general functions.
Here, we consider only algebraic varieties.
For these, Khovanskii bounds the sum \Blue{$b_*(X)$} of Betti numbers of
$X$ by a function that depends on $n$, the codimension of $X$, and the
total number of monomials appearing (with non zero coefficients) 
in polynomial equations defining $X$.

For real algebraic varieties $X$, the problem of bounding $b_*(X)$
has a long history.
O. A. Oleinik~\cite{O51} and J. Milnor~\cite{Mi64}
used Morse theory to estimate the number 
of critical points of a suitable Morse function to obtain a bound.
R. Thom~\cite{T64} used Smith Theory to bound the mod-2 Betti numbers.
This Smith-Thom bound has the form $b_*(X) \leq b_*(X_{\C})$, where
$X_{\C}$ is the complex variety defined by the same polynomials as $X$.
For instance, if $X \subset \R^n$ is defined by 
polynomials of degree at most $d$, then Milnor's bound is $b_*(X) \leq
d(2d-1)^n$.
If $X \subset {(\R \setminus \{0\})}^n$ is a smooth hypersurface
defined by a polynomial with Newton polytope $P$, then the Smith-Thom bound
is $b_*(X) \leq b_*(X_{\C})=n! \cdot \mbox{Vol}(P)$, where $\mbox{Vol}(P)$ is the usual
volume of $P$. 
We refer to~\cite{R92} for an informative history of the subject, and to the book~\cite{BR90}
for more details.

These bounds are given in term of degree or volume of a Newton polytope which are
numerical deformation invariants of the complex variety $X_{\C}$. 
In contrast, the topology of a real algebraic variety depends 
on the coefficients of its defining equations, 
and in particular, on the number of monomials involved in these equations.
For instance, Descartes's rule of signs implies that the number
of positive roots of a real univariate polynomial is less than its
number of monomials, but the number of complex roots is equal to its
degree.
Khovanskii's bound can be seen as a generalization of this Descartes
bound. 
It is smaller than the previous bounds when the defining equations
have few monomials compared to their degrees. 
While it has always been clear that Khovanskii's bounds
are unrealistically large, it appears very challenging to
sharpen them. 
Some progress has been made recently for the number of non degenerate
positive solutions to a system of $n$ polynomials in $n$
variables~\cite{BS}. 
Here, we bound $b_*(X)$, when $X$ is a fewnomial hypersurface.

Suppose that $X \subset \R^n_>$ is a smooth hypersurface defined by a Laurent polynomial 
with $n{+}l{+}1$ distinct monomial terms. 
Khovanskii~\cite{Kh91} (Corollary 4, p. 91)
showed  that
 \begin{equation}\label{eq:Khovanskii}
    b_*(X)\ \leq\   (2n^2-n+1)^{n{+}l} (2n)^{n-1} 2^{\binom{n{+}l}{2}}\,.
 \end{equation}

We use the new upper bound~\cite{BS} of $\frac{e^2+3}{4}
2^{\binom{l}{2}}n^l$ on the number of non degenerate positive solutions to a system of $n$
polynomials in $n$ variables having $n{+}l{+}1$ monomial terms, together with
stratified Morse theory for a manifold with corners~\cite{GM88} to give a new bound for 
the sum $b_*(X)$ of the Betti numbers of a fewnomial hypersurface.
Fix positive integers $N\geq n$ and $l$.

\begin{theorem}\label{Th:Betti}
  Let $X$ be a hypersurface in $\R^N_>$
  defined by a polynomial with $n{+}l{+}1$ monomials whose exponent
  vectors have
  affine span of dimension $n$.
  Then
\[
    b_*(X) \ <\   \frac{e^2+3}{4} 2^{\binom{l}{2}}
    \cdot \sum_{i=0}^{n} \tbinom{n}{i} i^l\ .
 \]
\end{theorem}

The bound of Theorem~\ref{Th:Betti} is bounded above by the simpler
expression
 \begin{equation}\label{eq:smaller}
  (e^2+3) 2^{\binom{l}{2}}\, n^l \cdot 2^{n-3}\,,
 \end{equation}
which is smaller than the bound on the number of connected
components (zeroth Betti number) proven in~\cite{BRS}.

The idea of the proof is as follows. 
If $X\subset \R^n$ is compact,
then a generic monomial function will be a Morse function on $X$. 
The critical points
of this Morse function are positive solutions to a polynomial system with the same initial
monomials and we may apply the bound~\cite{BS}.
If $X$ is non compact, consider its image $Z=\Log(X)$ under the
coordinatewise logarithmic map (a homeomorphism) 
and take the intersection with a sufficiently large simplex $\Delta$
so that $b_*(Z)=b_*(Z \cap \Delta)$.
The intersection of $Z$ with faces of
$\Delta$ are parts of fewnomial hypersurfaces about which we have a control
on the number of variables and monomials. 
The singular space $Z \cap \Delta$ is stratified 
by these hypersurfaces in the faces of $\Delta$.
In fact $Z \cap \Delta$ is a manifold with corners. 

Stratified Morse theory as developed by M.~Goresky and R.~MacPherson
extends classical Morse theory to compact stratified spaces. 
While it is complicated to apply this theory in general,
this is quite simple for manifolds with corners.
A stratified Morse function is a function whose restrictions to strata
are usual Morse functions. 
By stratified Morse theory, $b_*(Z\cap\Delta)$ is bounded from above
by the total number of critical points of the restrictions to the
strata of a stratified Morse function. 
We use a generic linear function as a stratified Morse function on $Z \cap \Delta$.
Such a linear function comes from a monomial function on $X$, and the number of critical points
for each stratum may be estimated by the bound~\cite{BS} as in the compact case.

Khovanskii's bound~\eqref{eq:Khovanskii} is a special case of more
general bounds which he obtains for fewnomial complete intersections.
It remains an important open problem to adapt the methods
of~\cite{BS} to fewnomial complete intersections.
We remark that our results hold for polynomials with real-number exponents.
In fact, our proofs (and the proofs in~\cite{BS}) involve this added
generality. 
We first describe Morse theory for a manifold with corners in Section 1
and prove
Theorem~\ref{Th:Betti} in Section 2.

We would like to thank Jean-Jacques Risler for providing us useful historical references on
the subject.

%
%
\section{Morse theory for a manifold with corners}

In classical Morse theory, the topology of a compact differentiable 
manifold $X$ is inferred from critical points of a sufficiently general
smooth function $f\colon X\to\R$, called a Morse function.
For example, $b_*(X)$ is bounded above by
the number of critical points of any Morse function.
Goresky and MacPherson~\cite{GM88} develop a version of Morse theory for 
stratified spaces.
This is particularly simple for manifolds with corners, which for us will 
be the intersection of our fewnomial hypersurface with a 
large simplex whose every face meets it transversally.

We begin with sketches of classical and of stratified Morse theory, 
and then explain how stratified Morse theory applies to a manifold with corners.
For further discussion and proofs see~\cite{GM88}.

A smooth function $f\colon X\to\R$ on a compact differential manifold is
a \Blue{{\it Morse function}} if its critical values (in $\R$)  are distinct and each critical 
point (in $X$) of $f$ is non degenerate (the Hessian matrix of second 
partial derivatives has full rank).
This implies that the critical points are discrete and there
are finitely many of them.
For each $c\in \R$, set $X_c:=f^{-1}(-\infty,c]$.
If $c$ is smaller than any critical value, then $X_c$ is empty, and if
$c$ is greater than all critical values, then $X=X_c$.
The Morse Lemmata
describe how the topological type of $X_c$ changes as 
$c$ increases from $-\infty$ to $\infty$.
The first Morse Lemma asserts that the topological type of $X_c$ is constant
for all $c$ lying in an open interval that contains no critical values.
The second Morse Lemma asserts that if $(a,b)$ contains a unique critical value $c=f(p)$,
then the pair $(X_b, X_a)$ is homeomorphic to the pair 
$(D^\lambda\times D^{n-\lambda},(\partial D^\lambda)\times D^{n-\lambda})$.
Here, $X$ has dimension $n$, $D^m$ is a closed disc of dimension $m$, and $\lambda$ is the
number of negative eigenvalues of the Hessian matrix of $f$ at $p$.
The long exact sequence of a pair  and induction on the critical values $c$ shows that
the sum of the Betti numbers of $X$ is bounded above by the number of critical points of
$f$.\smallskip 

Suppose now that $X$ is a Whitney stratified space, 
which we assume is embedded in an ambient manifold $W$.
A smooth function $f\colon X\to \R$ is the restriction to $X$ of a smooth 
function on $W$.
A \Blue{{\it critical point}} of $f$ is a critical point of the restriction of $f$ to any
stratum. 
(Each stratum in a Whitney stratified space is a manifold.)
A \Blue{{\it Morse function}} $f\colon X\to \R$ is a smooth function on $X$ whose critical
  values are distinct, and at each critical point $p$ of $F$, the restriction of $f$ to
  the stratum  containing $p$ is non degenerate at $p$.
There is a third condition that the differential of $f$ at $p$ does not annihilate
any limit of tangent spaces to any stratum other than the stratum containing $p$.

In stratified Morse theory, the first Morse Lemma holds as before and the second
Morse Lemma is modified as follows.
Let $p$ be a critical point of $f$ lying in a stratum $S$ of $X$.
Then let $D(p)$ be a small disk in $W$ transversal to the stratum $S$ 
such that $D(p) \cap S=\{p\}$ and call its 
intersection $N(p)$ with $X$ the \Blue{{\it normal slice}} to $S$ at $p$.
\Blue{{\it Normal Morse data}} at $p$ are a pair $(A,B)$, where 
$A$ is the set of points $x$ in $N(p)$ for which $|f(x)-f(p)|\leq\epsilon$ and 
$B$ is that part of the boundary of $A$ where $f(x)=f(p)-\epsilon$, 
and $\epsilon$ is any sufficiently small positive number.
The \Blue{{\it tangential Morse data}} at $p$ are the pair 
$(D^\lambda\times D^{n-\lambda},(\partial D^\lambda)\times D^{n-\lambda})$
appearing in the second Morse Lemma for the Morse function $f$ restricted to the stratum
$S$, which is a manifold of dimension $n$.
The second Morse Lemma in stratified Morse theory asserts that if 
the interval $(a,b)$ contains a unique critical value $c=f(p)$,
then the pair $(X_b, X_a)$ is homeomorphic the product of pairs
 \begin{equation}\label{Eq:SMT}
   (\mbox{normal Morse data at $p$}) \times
   (\mbox{tangential Morse data at $p$})\,.
 \end{equation}
(Recall that the product $(A,B)\times(A',B')$ of pairs is the pair
$(A\times A',\; A\times B'\cup A'\times B)$.)
In general, we must have detailed information about the interaction between the Morse
function and the stratification to use stratified Morse theory.

Such detailed information is available for
manifolds with corners.
Let $\R_\geq:=[0,\infty)$ be the non negative real numbers and $\R^m_\geq$ be the
  non negative orthant in $\R^m$.
A \Blue{{\it manifold with corners}} is a compact topological space $X$ with a covering by
charts, each homeomorphic to $\R^m_\geq\times\R^n$, where $m+n$ is the dimension of $X$.
Each point of $X$ has a well-defined tangent space isomorphic to $\R^{m+n}$,
{\it and} has a \Blue{{\it tangent cone}} isomorphic to $\R^m_\geq\times\R^n$.
Points with tangent cone isomorphic to $\R^m_\geq\times\R^n$ for $m$ fixed form a
submanifold of dimension $n$.
All such submanifolds form the \Blue{{\it boundary strata}} of $X$.

Suppose that $f\colon X\to \R$ is a Morse function.
Let $p$ be a point in a boundary stratum locally homeomorphic to $\R^m_\geq\times\R^n$.
The normal slice $N(p)$ to such a point is a neighborhood of the origin in the cone
$\R^m_\geq$. 
Due to the third condition on Morse functions, 
there are exactly two possibilities for the normal Morse data $(A,B)$ at $p$.
The first component $A$ is homeomorphic to the $m$-simplex
\[
   \Delta_m\ :=\ \{x\in\R^m_\geq : |x|:=x_1+\dotsb+x_m\leq 1\}
   \quad(\simeq\ D^m)
\]
and the second component $B$ is either the empty set $\emptyset$
(if $f(p)$ is locally the minimum value of $f$ on $N(p)$) 
or a contractible subset of the $(m{-}1)$-simplex $|x|=1$ otherwise.
(If $f(p)$ is a local maximum, then $B$ is the full $(m{-}1)$-simplex.)
We illustrate this in Figure~\ref{F:Normal}.
\begin{figure}[htb]
\[
\begin{picture}(350,115)(-10,-10)
   \put( 35,56){$A$}   \put(-10,53){\Brown{$p$}}
   \put(  0,51){\includegraphics{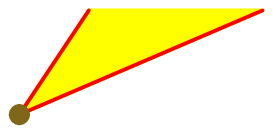}}
   \put( 30,90){\Blue{$B=\emptyset$}}

   \put(143,90){$A$}   \put(95,53){\Brown{$p$}}
   \put(105,21){\includegraphics{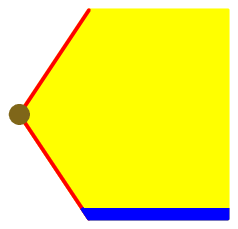}}
   \put(143,12){\Blue{$B$}}

   \put(245,48){$A$}   \put(202,53){\Brown{$p$}}
   \put(210,21){\includegraphics{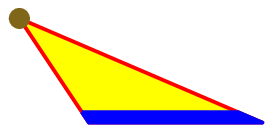}}
   \put(250,12){\Blue{$B$}}

   \put(315,0){\includegraphics{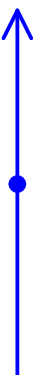}}
   \put(140,-13){$f$}                 \put(325,85){$\R$}
   \put(125,-2){\vector(1,0){40}}
  \end{picture}
\]
\caption{Normal Morse data}
\label{F:Normal}
\end{figure}

Observe that in the case when $B=\emptyset$, we have 
 \[
    (A,B)\times(D^\lambda\times D^{n-\lambda},(\partial D^\lambda)\times D^{n-\lambda})
     \ \simeq\ 
   (D^\lambda\times D^{m+n-\lambda},(\partial D^\lambda)\times D^{m+n-\lambda})
 \]
Thus this critical point could contribute to the sum of the Betti numbers of $X$.
On the other hand, when $B\neq\emptyset$, the pair 
$(B,B)$ is a deformation retract of $(A,B)$, and so the critical point $p$
does not contribute to the sum of the Betti numbers of $X$, as the topology of $X_a$ does
not change as $a$ passes $f(p)$.

\begin{proposition}\label{prop:Morse_bounds}
 The sum of the Betti numbers of a manifold $Z$ with corners is at most the number of
 critical points $p$ of a Morse function $f$ for $Z$ where the minimum of $f$ on the
 normal slice to $Z$ at $p$ is attained at $p$.
\end{proposition}

\begin{example}
 Proposition~\ref{prop:Morse_bounds} is illustrated by the cannoli shell, which 
 is a manifold with corners.
 It is a cylinder $D^1\times S^1$ having two boundary strata,
 each homeomorphic to the circle $S^1$.  
 The height function of Figure~\ref{Fig:cannoli} is a Morse function for the cannoli
 shell with four critical points.
 Only the two smallest critical values contribute to its topology and its Betti numbers,
 by Proposition~\ref{prop:Morse_bounds}.
\end{example}

\begin{figure}[htb]
\[
  \begin{picture}(160,150)(-5,0)
    \put(  0,  0){\includegraphics[height=160pt]{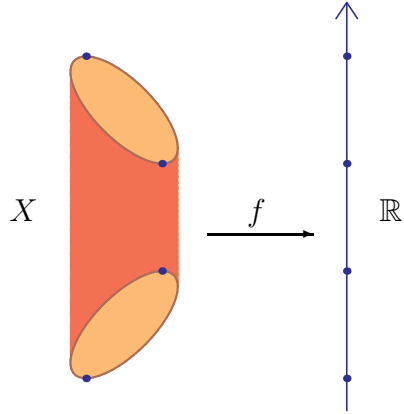}}
    \put(  0, 70){$X$}    \put(140,70){$\R$}
    \put( 90, 70){$f$}
    \put( 75, 65){\vector(1,0){40}}
  \end{picture}
\]
\caption{Morse function on the cannoli shell}\label{Fig:cannoli}
\end{figure}

%
%
\section{Proof of Theorem~\ref{Th:Betti}}

Let $f(x_1,\dotsc,x_N)$ be a real Laurent polynomial with $n{+}l{+}1$ exponent vectors that
span an $n$-dimensional affine subspace of $\R^N$ such that
$\Blue{X}:=\calV(f)\subset\R^N_>$ is a smooth hypersurface.
Rather than work with polynomials in $\R^N_>$, we work with exponential sums in $\R^N$.
These notions are related via a logarithmic change of coordinates.
Consider the isomorphisms of Lie groups.
\[
  \begin{array}{rclcrcl}
   \mbox{Exp}\colon\ \R^N&\longrightarrow&\R^N_>&\qquad&
   \Log\colon\ \R^N_>&\longrightarrow&\R^N\\\rule{0pt}{14pt}
   (z_1,\dotsc,z_N)&\longmapsto&(e^{z_1},\dotsc,e^{z_N})&\qquad&
   (x_1,\dotsc,x_N)&\longmapsto&(\log(x_1),\dotsc,\log(x_N))
  \end{array}
\]
Under this isomorphism monomials $x^\alpha$ correspond to exponentials $e^{z\cdot\alpha}$,
and so the fewnomial $f=\sum_i c_i x^{\alpha_i}$ corresponds to the exponential sum
\[
   \Blue{\varphi}\ :=\ \sum_{i=0}^{n+l} c_i e^{z\cdot\alpha_i}\ .
\]
Let $\Blue{Z}:=\calV(\varphi)\subset\R^N$ be the hypersurface defined by $\varphi$, which
is homeomorphic to $X$.
For exponential sums, it is quite natural to allow the exponents $\alpha_i$ to be real vectors.
We will prove Theorem~\ref{Th:Betti} in these logarithmic coordinates
and with real exponents.\medskip

\noindent{\bf Theorem~\ref{Th:Betti}$'$.}
{\it
  The sum of the Betti numbers of a hypersurface in $\R^N$ defined by 
  an exponential sum with $n{+}l{+}1$ terms whose exponent vectors have
  affine span of dimension $n$ is at most 
\[
  \frac{e^2+3}{4} 2^{\binom{l}{2}}
    \cdot \sum_{i=0}^{n} \tbinom{n}{i} i^l\ .
 \]
}\medskip

Multiplying $\varphi$ by $e^{-z\cdot\alpha_0}$, we may assume that $\alpha_0=0$.
Then $\alpha_1,\dotsc,\alpha_{n{+}l}$ span an $n$-dimensional linear subspace of $\R^N$.
After a linear change of coordinates, we may assume that $\varphi$ only involves the first
$n$ variables, and thus the hypersurface $Z$ becomes a cylinder
\[
   Z\ \simeq\ \{ z\in \R^n : \varphi(z)=0\}\times\R^{N-n} \,.
\]
Thus it suffices to prove Theorem~\ref{Th:Betti}$'$ when $N=n$.

Since the exponents $\alpha_1,\dotsc,\alpha_{n+l}$ span $\R^n$, we may assume that 
the first $n$ are the standard unit basis vectors in $\R^n$, and thus $\varphi$ includes
the coordinate exponentials $e^{z_i}$ for $i=1,\dotsc,n$.
Let $M:=(M_0,M_1,\dotsc,M_n)$ be a list of positive numbers and set 
\[
  \Blue{\Delta_M}\ :=\ \{z\in\R^n : z_i\geq -M_i, \ i=1,\dotsc,n
    \quad\mbox{and}\quad \sum_iz_i\leq M_0\}\,,
\]
which is a non empty simplex.
We will use stratified Morse theory to bound the Betti numbers of 
$\Blue{Y}:=Z\cap \Delta_M$ when $M$ is general.

\begin{theorem}\label{Th:bound}
  For $M$ general, the sum of the Betti numbers of $Y$ is at most
\[
  \frac{e^2+3}{4} 2^{\binom{l}{2}}
    \cdot \sum_{i=0}^{n} \tbinom{n}{i} i^l\ .
 \]
\end{theorem}

\begin{proof}[Proof of Theorem~$\ref{Th:Betti}'$]
For any $r>0$, set
$$
  \Blue{Z_r}\ :=\ \{ z \in Z :  \| z\|  <r \}\,.
$$
By~\cite[Corollary 9.3.7]{BCR}, there is some $R>0$ such that if
$r \geq R$ then $Z_r$ is deformation retract of $Z$ and
$Z_R$ is a deformation retract of $Z_r$.

 Choose $M$ and $\Blue{r}>R$ so that $\Delta_M$ is sandwiched between the balls of
 radius $R$ and $r$ centered at the origin.
 Let $\Blue{\rho}\colon Z_r\to Z_R$ be the retraction.
 We have the maps 
\[
   Z_R\ \hookrightarrow\ Y=Z\cap\Delta_M\ 
        \hookrightarrow\ Z_r\ \xrightarrow{\ \rho\ }\ Z_R\,,
\]
 whose composition is the identity.
 The induced maps on the $i$th homology groups,
\[
   H_i(Z_R)\ \longrightarrow\ H_i(Y)\ \longrightarrow\ 
   H_i(Z_r)\ \xrightarrow{\ \rho_*\ }\ H_i(Z_R)\,,
\]
 have composition the identity.
 This gives the inequality
\[
    \dim H_i(Y)\ \geq\ \dim H_i(Z_R)\ =\ \dim H_i(Z)\,.
\]
 Summing over $i$ shows that Theorem~\ref{Th:Betti}$'$ is a consequence of 
 Theorem~\ref{Th:bound}.
\end{proof}

\begin{proof}[Proof of Theorem~$\ref{Th:bound}$]
 Given positive numbers $M=(M_0,M_1,\dotsc,M_n)$, define affine hyperplanes in
 $\R^n$ 
\[
   \Blue{H_0}\ :=\ \{z : {\textstyle \sum_i} z_i=M_0\}
   \qquad\mbox{and}\qquad
   \Blue{H_i}\ :=\ \{z : z_i=-M_i\}\,,\quad\mbox{for}\ i=1,\dotsc,n\,.
\]
 For each proper subset $S\subset\{0,\dotsc,n\}$, define an affine linear subspace
\[
   \Blue{H_S}\ :=\ \bigcap_{i\in S} H_i\,.
\]
 Since each $M_i>0$, this has dimension $n-|S|$, and these subspaces are the 
 affine linear subspaces supporting the faces of the simplex $\Delta_M$.

 Choose $M$ generic so that for all $S$ the subspace $H_S$ meets $Z$
 transversally.
 For each $S$, set $\Blue{Z_S}:=Z\cap H_S$. 
 This is a smooth hypersurface in $H_S$ and therefore has
 dimension $n-|S|-1$.
 The boundary stratum $Y_S$ of $Y=Z\cap \Delta_M$ lying in the face supported by $H_S$ is an
 open subset of $Z_S$.

  For a non zero vector $u\in \R^n$, the directional derivative $D_u\varphi$ is
\[
   \sum_{i=1}^{n{+}l}  (u\cdot\alpha_i) c_i e^{z\cdot\alpha_i}\ ,
\]
 which is an exponential sum having the same exponents as $\varphi$.
 Let \Blue{$L_u$} be the linear function on $\R^n$ defined by $z\mapsto u\cdot z$.

 The critical points of the function $L_u$ restricted to $Z$ are the zeroes of the system
\[
    \varphi(z)\ =\ 0
    \qquad\mbox{and}\qquad
    D_v\varphi(z)\ =\ 0\quad\mbox{for }v\in u^\perp\,.
\]
 When $u$ is general and we choose a 
%
%
 basis for $u^\perp$, this becomes a system of $n$
 exponential sums in $n$ variables, all involving the same $n{+}l{+}1$ exponents.
 By the fewnomial bound in~\cite{BS}, the number of solutions is at most
\[
   \frac{e^2+3}{4} 2^{\binom{l}{2}} n^l\,.
\]

 We use this to estimate the number of critical points of the function $L_u$ restricted to
 $Z_S$.
 The restriction of $\varphi$ to $H_S$ defines $Z_S$ as a hypersurface in
 $H_S$. 
 We determine this restriction.
 If $i\in S$ with $i>0$ then we may use the equation $z_i=-M_i$ to eliminate the variable
$z_i$ and the exponential $e^{z_i}$ from $\varphi$.
 If $0\in S$, then we choose $j\not\in S$ and use the equation $\sum_iz_i=M_0$ to
 eliminate $z_j$ from $\varphi$.
 Let \Blue{$\varphi_S$} be the result of this elimination.
 It is an exponential sum in $n-|S|$ variables and its number of terms is at most
\[
  \begin{array}{rl}
   (n-|S|)+l+1&\mbox{ if } 0\not\in S\\
   (n-|S|)+(l+1)+1&\mbox{ if } 0\in S\rule{0pt}{18pt}
  \end{array}
\]

Thus if $u$ is a general vector in $\R^n$, then the 
 number of critical points of the linear function $L_u|_{H_S}$ on $Z_S$ is at most
\[
  \begin{array}{rl}
   \frac{e^2+3}{4}2^{\binom{l}{2}} (n-|S|)^l&\mbox{ if } 0\not\in S\\
   \frac{e^2+3}{4}2^{\binom{l+1}{2}}
(n-|S|)^{l+1}&\mbox{ if } 0\in S\rule{0pt}{18pt}
  \end{array}
\]
 We use this estimate and stratified Morse theory to bound the Betti numbers of $Y$.

 Let $u$ be a general vector in $\R^n$ such that $L_u$ is a Morse function for the
 stratified space $Y$.
 By Proposition~\ref{prop:Morse_bounds}, the sum of the Betti numbers of 
 $Y$ is bounded by the number of critical points $p$ of $L_u$ for which 
 $L_u$ achieves its minimum on the normal slice $N(p)$ at $p$.
 Since the strata $Y_S$ of $Y$ are open subsets of the manifolds $Z_S$, this number is
 bounded above by the number of such critical points of $L_u$ on the manifolds $Z_S$.
 We argue that we can alter $u$ so that no critical point in any $Z_S$ with $0\in S$
 contributes.
 That is, change $u$ so that for no critical point $p$ in a stratum $Z_S$ with 
 $0\in S$ is $p$ the local minimum of $L_u$ on the normal slice $N(p)$.

 Suppose that $p$ is a critical point of $L_u$ on a stratum $Z_S$ with $0\in S$.
 Then $p$ lies in $H_0$, and so the linear function $L_{(1,\dotsc,1)}$ restricted
 to the normal slice $N(p)$ at $p$ takes its maximum value $M_0$ at $p$.
 If we replace $u$ by $u+\lambda (1,\dotsc,1)$ we change $L_u$ on $H_S$ by the constant
 $\lambda M_0$ and $p$ will still be a critical point for $L_{u+\lambda (1,\dotsc,1)}$ on
 $Z_S$.
 If $\lambda$ is sufficiently large, then $L_{u+\lambda (1,\dotsc,1)}$ 
 does not achieve its  minimum value on $N(p)$ at $p$.

 There are finitely many critical points $p$ of $L_u$ on strata $Z_S$ with $0\in S$.
 Hence, there is a positive number $\lambda$ so that at each of these critical
 points $p$,  $L_{u+\lambda (1,\dotsc,1)}$ 
 does not achieve its
 minimum value on $N(p)$ at $p$.
 We may further choose $\lambda$ so that $L_{u+\lambda (1,\dotsc,1)}$ is a Morse function
 for the stratified space $Y$.
 Since only the critical points on strata $Z_S$ with $0\not\in S$ can contribute to the
 Betti numbers of $Y$, we see that its sum of Betti numbers is bounded above by
 \[
   \frac{e^2+3}{4} 2^{\binom{l}{2}}
      \sum_{S\subset\{1,\dotsc,n\}} (n-|S|)^l
   \ =\ 
   \frac{e^2+3}{4} 2^{\binom{l}{2}}
      \sum_{i=0}^n  \binom{n}{i} (n-i)^l\,.
 \]
 Since $\binom{n}{i}=\binom{n}{n-i}$, we replace $i$ by $n-i$ to complete the proof of
 Theorem~\ref{Th:bound}. 
\end{proof}

We remark that the idea of cutting a fewnomial variety with
a monomial hypersurface, which preserves the fewnomial structure,
is not ours, but may be found in papers of Rojas~\cite{Rojas} and Perrucci~\cite{Pe05},
who used this to bound the number of connected components of a fewnomial variety.

We deduce the formula~\eqref{eq:smaller} from Theorem~\ref{Th:Betti}.
Observe that 
\[
      \sum_{i=0}^n  \binom{n}{i} i^l\ =\ 
   n^l\ \sum_{i=0}^n  \binom{n}{i} \left(\frac{i}{n}\right)^l\,,
\]
and thus the sum is a decreasing function of $l$.
When $l=1$, this sum is 
\[
   \sum_{i=0}^n  \binom{n}{i} \frac{i}{n}\ =\ 
   \sum_{i=1}^n  \binom{n-1}{i-1} \ =\ 2^{n-1}\,.
\]
Substituting this into the formula of Theorem~\ref{Th:Betti} gives~\eqref{eq:smaller}.


\providecommand{\bysame}{\leavevmode\hbox to3em{\hrulefill}\thinspace}
\providecommand{\MR}{\relax\ifhmode\unskip\space\fi MR }
\providecommand{\MRhref}[2]{%
  \href{http://www.ams.org/mathscinet-getitem?mr=#1}{#2}
}
\providecommand{\href}[2]{#2}

\end{document}